\documentclass[11pt]{article}

\usepackage{amssymb,latexsym}

\usepackage[dvips]{graphics}

\usepackage{epsfig}
\usepackage{color}

\usepackage{enumerate}
\usepackage{amsfonts}
\usepackage{graphicx}
\usepackage{amsmath}
\usepackage{amsthm}
\usepackage{amssymb}
\usepackage[mathscr]{euscript}

\usepackage{natbib}

\bibpunct{(}{)}{,}{a}{}{;}

\usepackage[margin=2.3cm]{geometry}

\newcommand{\PP}{\mathbf{P}}
\newcommand{\EE}{\mathbf{E}}

\newcommand{\as}{\mbox{\hspace{.3cm} a.s.}}

\newcommand{\nid}{\noindent}
\newcommand{\wt}{\widetilde}
\newcommand{\var}{\text{Var}}
\newcommand{\cov}{\text{Cov}}

\newcommand{\B}[1]{\textbf{#1}}
\newcommand{\C}[1]{\mathcal{#1}}
\newcommand{\D}[1]{\mathbb{#1}}

\newcommand{\ga}{\alpha}
\newcommand{\gb}{\beta}
\newcommand{\gga}{\gamma}
\newcommand{\gd}{\delta}
\newcommand{\gep}{\varepsilon}
\newcommand{\gz}{\zeta}

\newcommand{\gl}{\lambda}
\newcommand{\gs}{\sigma}

\newcommand{\ol}[1]{\overline{#1}}
\newcommand{\ul}[1]{\underline{#1}}

\newtheorem{theorem}{Theorem}
\newtheorem{proposition}{Proposition}
\newtheorem{corollary}{Corollary}
\newtheorem{lemma}{Lemma}

\theoremstyle{remark}

\makeatletter
\def\blfootnote{\gdef\@thefnmark{}\@footnotetext}
\makeatother

\begin{document}

\title{Metastable states  \\ in Brownian energy landscape}

\author{Dimitris Cheliotis}

\author{
\renewcommand{\thefootnote}{\arabic{footnote}}
Dimitris\ Cheliotis
\footnotemark[1]
}

\footnotetext[1]{
{Department of Mathematics, University of Athens,
Panepistimiopolis, 15784 Athens, Greece}, 
{\sl dcheliotis@math.uoa.gr}
}

\blfootnote{
 MSC2010 Subject Classifications: 60K37, 60G55, 60F05}
 

\blfootnote{This research has been co-financed by the European Union and Greek
national funds through the Operational Program "Education and Lifelong Learning" of the National Strategic Reference Framework (NSRF), (ARISTEIA I, MAXBELLMAN 2760).}

\date{August 13, 2013}

\maketitle

\begin{abstract} Random walks and diffusions in symmetric random
environment are known to exhibit metastable behavior: they tend to
stay for long times in wells of the environment. For the case that the
environment is a one-dimensional two-sided standard Brownian motion,
we study the process of depths of the consecutive wells of increasing
depth that the motion visits. When these depths are looked in
logarithmic scale, they form a stationary renewal cluster process. We
give a description of the structure of this process and derive from it
the almost sure limit behavior and the fluctuations of the empirical
density of the process.
\end{abstract}

\section{Introduction and statement of the results}

Consider $(X_t)_{t\ge0}$ Brownian motion with drift in $\D{R}$, starting from 0, with the drift at each point $x\in\D{R}$
being $-f'(x)/2$ for a certain differentiable function $f$.
That is, $(X_t)_{t\ge0}$ satisfies the SDE
$$dX_t=d\gb_t-\frac{1}{2}f'(X_t) dt,$$
with $\gb$ a standard Brownian motion. This is called diffusion in the
environment $f$, and it has $e^{-f(x)}\ dx$ as an invariant measure. In
statistical mechanics terms, $f$ gives the energy profile, and the above SDE
defines the Langevin dynamics for the corresponding measure $e^{-f(x)}\ dx$.
The diffusion likes to go downhill on the environment $f$, decreasing the
energy, and thus it tends to stay around local minima of $f$.
If the set $M_f$ of local minima of $f$ is
non-empty, the diffusion exhibits metastable behavior,
with metastable states being the points of $M_f$ (see \citet{BO}, Section 8).

Now, for each point $x_0$
of local minimum, there are intervals $[a,c]$ containing
$x_0$ with the property that $f(x_0)$ is the minimum value of $f$ in $[a, c]$
and $f(a), f(c)$ are the maximum values of $f$ in the intervals $[a, x_0], [x_0,
c]$ respectively.
Let $J(x_0):=[a_{x_0}, c_{x_0}]$ be
the maximal such interval. This is the ``interval of influence''
for $x_0$.
 We call
$f|J(x_0)$ the \textbf{well} of $x_0$, the number
$\min\{f(a_{x_0})-f(x_0), f(c_{x_0})-f(x_0)\}$ the \textbf{depth}
of the well, and $x_0$ the \textbf{bottom} of the well. If the diffusion starts
inside $J(x_0)$,
typically it is trapped in that interval for a time that depends
predominantly on the depth of the well.

Also, for $h>0$, we say that the local minimum $x_0$ is
a point of \textbf{$h$-minimum} for $f$ if the depth of its well is at least $h$, while a point $x_0$ is called a point of \textbf{$h$-maximum} for $f$ if it is a point of $h$-minimum for $-f$.

A case of particular interest is the one where the function $f$ above is a ``typical'' two sided Wiener path with $f(0)=0$. Of course, an $f$ picked from the Wiener
measure is not differentiable, but there is a way to make sense of the above SDE
 defining $X$ through a time and space transformation. See \citet{SHI} for the
construction.

From now on, we will denote the two sided Wiener path with $B$.
Due to the nature of a typical Wiener path, once the diffusion exits an interval $J(x_0)$, it is trapped in another well. We will define a process $x_B$ that records some local minima of the path of $B$ in the order that are visited by a typical
diffusion path, but not all of them. Roughly, assuming that the value of the process at some point is $x_0$, its next
value is going to be the unique local minumum $x_1$ whose interval of influence is the smallest one satisfying $J(x_1)\supsetneqq J(x_0)$. The well $B|J(x_1)$ is the minimal one containing strictly $B|J(x_0)$, it is the first well right after
$J(x_0)$ that can trap the diffusion for considerably more time, and this because it has greater depth.

The formal definition of the process $x_B$ goes as follows. With probability one, for all $h>0$, there are $z_{-1}(h)<0<z_1(h)$ points of $h$-extremum ($h$-mimimum or $h$-maximum) for $B$ closest to zero from the left and right respectively. Exactly one of them is a point of
$h$-minimum for $B$. This we denote by $x_B(h)$.

The process $(x_B(h))_{h>0}$  has piecewise constant paths, it is
left continuous, and there are several results showing its impact
on the behavior of the diffusion. For example, $X_t-x_B(\log t)$
converges in distribution as $t\to+\infty$ (\citet{TA}), i.e., $x_B$ gives a
good prediction for the location $X_t$ of the diffusion at large times.
 Note also that, by Brownian scaling, for $a>0$ the process $x_B$ satisfies
\begin{equation}\label{xBscaling}
(x_B(ah))_{h>0}\overset{d}{=}(a^2x_B(h))_{h>0}.
\end{equation}

We would like to study the set of points where $x_B$ jumps, because this shows
how frequently the diffusion discovers the bottom of a well that is deeper than
any well encountered by then. It turns out that it is more
convenient to consider this set in logarithmic scale, that is,
the point process
$$\xi:=\{t\in\D{R}: x_B \mbox{ has a jump at } e^t\}.$$
The purpose of this work is to describe the structure of $\xi$.
A crucial observation is that the law of $\xi$ is translation invariant
because of the scaling relation \eqref{xBscaling} for
$x_B$. Since $B$ is continuous, the set $\xi$ has no finite accumulation
point.

For any set $A$ define $N(A):=|\xi\cap A|$, the cardinality of
$\xi\cap A$, i.e., $N$ is the counting measure induced by $\xi$. When
$A$ is an interval, we  will write $NA$ instead of $N(A)$.

The following result (Theorem 2.4.13 in \citet{ZE}) gives the probability that
$\xi$ does not hit an interval.
\begin{theorem}[Dembo, Guionnet, Zeitouni] \label{DGZThm}
For $t>0$,
\begin{equation}\PP(N[0,t]=0)=\frac{1}{t^2}\left(\frac{5}{3}-\frac{2}{3}e^{1-t}
\right).
\end{equation}

\end{theorem}
This allows us to compute the mean density, $\EE N(0, 1]$, of the process, because for a simple
stationary point process, its mean density equals also its intensity $\lim_{t\to
0^+} t^{-1}\PP(N(0,t]>0)$ (Proposition 3.3 IV in \citet{DVJ}). Thus we
get the following result, which has been predicted
by physicists (relation (84) in \citet{DMF}) via
renormalization arguments.
\begin{corollary}[Mean density] \label{frequency} For every Borel
set
$A\subset\D{R}$,
$\EE N(A) =\frac{4}{3} \gl (A)$, where $\gl$ is
Lebesgue measure. Moreover,
\begin{equation} \label{meanDensity}
 \lim_{t\to \infty} \frac{N[0,t]}{t} = \frac{4}{3} \as
\end{equation}
\end{corollary}
\nid In Section \ref{elementaryProof}, we give an easy proof of this
corollary  which avoids the use of Theorem \ref{DGZThm}. 

Combining this with well known localization results for the
diffusion, we infer that the diffusion jumps to a deeper well
extremely rarely, at times that progress roughly as $\exp(\exp(3n/4))$. We also remark that for
the process $\hat \xi:=\{t: x_B \text{ changes sign at } e^t \}$, which is a subset of $\xi$, it was
shown in \citet{C} that it has mean density 1/3. On
average, one in every
four consecutive jumps is a sign change.

The description of $\xi$ given in the coming subsection has the following implication. 

\begin{theorem}[Fluctuations] \label{fluctuations}
As $t\to\infty$, the following convergence in distribution holds:
\begin{equation}\frac{1}{\sqrt{t}}\left(N[0, t]-\frac{4}{3}t\right) \Rightarrow
\mathcal{N}(0, \gs^2),
\end{equation}
with 
$\sigma^2=
\frac{64}{27}-\frac{4}{9}\int_0^\infty e^{-t}(1+t)^{-1}\, dt\approx 2.105327 $
\end{theorem}

\subsection{The structure of the process $\xi$} \label{structureSSection}

$\xi$ is a renewal cluster process in $\D{R}$. That is, it consists of: 

(i) a skeleton of points that serve as ``centers'' of clusters, 

\nid together with

(ii) the cluster points.  

\nid The centers form a stationary renewal process in $\D{R}$. Then each
cluster is distributed in a certain way relative to its center (to be exact,
relative to the skeleton).

More specifically, let $\psi$ be a stationary renewal
process in $\D{R}$ with interarrival distribution that of the sum $W_1+W_2$ of
two independent random variables with $W_1\sim$ Exponential(1), $W_2\sim$
Exponential(2). $\psi$ is the ``centers'' process.

Next, we describe the law of a cluster with center at 0.

\nid Count the points of a Poisson point process in $[0, \infty)$ with rate 1 as $(t_k)_{\ge2}$ in
increasing order, and let $t_1=0$. Out of the points
$$t_1, t_2, \ldots$$
we will keep only the first $\C{N}$, where $\C{N}$ is defined as
follows. Take a sequence $(Y_i)_{i\ge1}$ of i.i.d. random variables, independent
of $(t_k)_{k\ge2}$, each with distribution Exponential(1). Define recursively a
sequence $(z_k)_{k\ge1}$ as follows:
\begin{align*}
z_1&:=1\\
z_{k+1}&:=z_k+Y_k\, e^{t_k} \text{ for $k\ge1$},
\intertext{and let} 
\C{N}&:=\max\{i:t_i\le\log z_i\},\\
\C{T}&:=\{t_1, t_2, \ldots, t_{\C{N}}\}.
\end{align*}
$\C{N}$ is finite with
probability 1 as we will see in Theorem \ref{MGFProposition}.

A cluster with center at 0 has the law of $\C{T}$. 

\nid Let also 
\begin{align*}
F&:=\log z_{\C{N}+1}. \\
\intertext{Independent of $(t_k)_{k\ge1}, (Y_i)_{i\ge1}$ take another random variable
$Z\sim$ Exponential(2), and let}
R&:=F+Z.
\end{align*}
 
 \begin{figure}[htbp]
 \begin{center}
  \resizebox{11cm}{!} {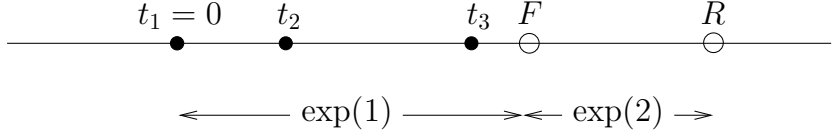} \caption{A typical cluster with
center at $t_1$. Points appear at an interval with length distribution $\exp(1)$. This cluster has 3 points, marked with black dots. The next cluster right to it will have its
center at $R$. The interval between $F$ and $R$ has length distribution $\exp(2)$, and it is not allowed to have points.}
 \label{fig1}
 \end{center}
 \end{figure}
\nid Note that $\C{T}\subset [0, F)$. We will see in Section \ref{InsideAnExcursion} that $F\sim
\text{Exponential}(1),$ while, by construction, $R-F\sim \text{Exponential}(2)$.

The role of $F$ and $R$ is the following. Given that $x$ is a point in the process of the centers,  the cluster at $x$ has
law $x+\C{T}$, while the next cluster to the right of it has center at $x+R$,
and thus distributed as $x+ R+\C{T}'$, with $\C{T}'$ an independent copy of
$\C{T}$.

And we are now ready to give the  formal description of $\xi$. 
For each $x\in\psi$ let $x^+:=\inf\{y\in \psi: y>x\}$, the nearest right neighbor of $x$ in $\psi$.

\begin{theorem} \label{ClusterForm} $\xi$ has the same law as
$$\bigcup_{x\in \psi}\big\{x+\C{T}_x(x^+-x)\big\},$$
where $\{\C{T}_x(x^+-x): x\in \psi\}$ are independent, and $\C{T}_x(x^+-x)$ is
distributed as $\C{T}$ given that $R=x^+-x$.
\end{theorem}

Finally, we look closer into the law of a cluster. The random variables
$\C{N}, F$ are positively correlated, and the following result captures their
joint distribution. 
For its statement, we will use the confluent
hypergeometric function of the second kind, which is usually  denoted by
$\Psi$. This has three arguments, and its value at a point $(x, y, z)$
is denoted by $\Psi(x, y; z)$.

\begin{theorem} \label{MGFProposition}
The moment generating function of $(\C{N}, F)$ equals 
\begin{equation}
\EE(e^{\gl \C{N}+\mu F })=e^{\gl}\frac{\Psi(1-e^{\gl},
1+\mu;1)}{\Psi(-e^{\gl}, \mu;1)}
\end{equation}
 for all $(\gl, \mu)\in\D{R}^2$ where the generating function is finite.  This
set of $(\gl, \mu)$
is open, convex,
and contains $(0, 0)$. In particular, $\EE(\C{N})=2$.
\end{theorem}

The main ingredient in the proof of the above results is a 
new way to follow the evolution of $x_B$, using excursion theory.
This point of view has also been useful in the study of large 
deviations for the family of paths $\{\gep x_B( \, \cdot\, 
):\gep>0\}$ as $\gep\to0$ (see \citet{CV}). Two other ways of 
studying $x_B$ have been exhibited in \citet{ZE} and \citet{DMF}.

Theorems \ref{fluctuations}, \ref{ClusterForm} and \ref{MGFProposition} are proven in Sections \ref{ProofCLT}, \ref{structuresection} and \ref{GFComputation} respectively. 
An alternative proof of Corollary \ref{frequency} is given in Section \ref{elementaryProof}, while Section \ref{expComputation} gives an elementary computation of a certain expectation which makes the proof of Corollary \ref{frequency} independent of Theorem \ref{MGFProposition}.

\section{Description of the process $\xi$. Proof of Theorem \ref{ClusterForm}}
\label{structuresection}

In this section, we study how the process $x_B$ evolves, and justify the
description of the structure of the process $\xi$ given in Section
\ref{structureSSection}, thus proving Theorem \ref{ClusterForm}. 

We will use elements of excursion theory, for which we refer the reader to
\citet{BE}, Chapter IV. \
For ease in exposition, when working with the excursions of a real valued process $(Y_t)_{t\ge0}$ away from 0,
 by the term ``actual domain'' of an excursion $\gep$ we will mean the interval $[c, d]$ in the domain of $Y$
 where the excursion happens and not $[0, d-c]$ or $[0, \infty)$, which are the two common conventions for the domain of $\gep$ in the literature (\cite{BE} adopts the first). Also we will abuse notation (notice the conflict with
\eqref{overline} below) and denote by $\ol{\gep}$, the \textbf{height} of
$\gep$, that is, the supremum of $\gep$ in its domain. 

For any process $(Y_t)_{t\in I}$ defined in an interval $I$ containing 0, we define the processes $\ul{Y}, \ol{Y}$ of the running 
infimum and supremum of $Y$ respectivelly as 
\begin{align}
\ul{Y}_t:=\inf\{Y_s: s
\text{ between $0$ and } t\}, \\
\ol{Y}_t:=\sup\{Y_s: s
\text{ between $0$ and } t\}  \label{overline} 
\end{align}
for all $t\in I$. This notation will be used throughout the paper.

Now let $(B_s)_{s\in\D{R}}$ be a two sided standard
Brownian motion. For $\ell>0$, we define
\begin{equation}
  \begin{aligned}
  \label{Theta}
H_\ell^-:=& \sup \{s<0: B_s=\ell\},\\
H_\ell^+:=& \inf \{s>0: B_s=\ell\}, \\
\Theta_\ell:=&-\min \{B_s: s\in[H_\ell^-, H_\ell^+]\}.
  \end{aligned}
\end{equation}
Following the path $B|[H_\ell^-, H_\ell^+]$ as $\ell$ increases reveals the
consecutive values of $(x_B(h))_{h>0}$  
in the same
order that the diffusion typically discovers them. Adopting this view, leads us
to
consider the processes  $\{e_t^+: t\ge0\}$ and $\{e_t^-: t\ge0\}$  of
excursions away from 0 of $(\ol{B}_s-B_s)_{s\ge0}$
and
$(\ol{B}_s-B_s)_{s\le 0}$ respectivelly. Both processes are
parametrized by the inverse of the local time processes $(\ol{B}_s)_{s\ge0}$ and
$(\ol{B}_s)_{s\le0}$ respectivelly, and of course they are independent and
identically distributed.

The continuity of $B$ implies that $\Theta$ is piecewise constant,
left continuous, and the set of points where it jumps, call it $\C{L}$, has 0 as
only accumulation
point. 
 \begin{figure}[htbp] 
 \begin{center}
 \resizebox{10cm}{!} {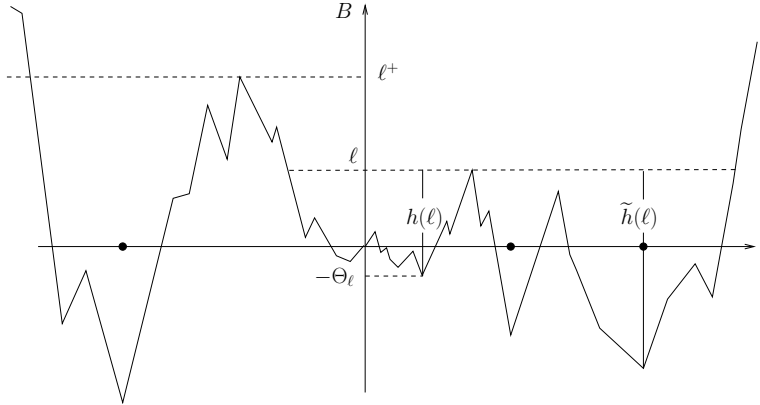} \caption{Following the evolution of
$x_B$. The dots mark three consecutive values of $x_B$. $\Theta$
jumps at the values $\ell, \ell^+$.}
 \label{fig2}
 \end{center}
 \end{figure}

\nid Pick $\ell\in\C{L}$. With probability 1, exactly one of 
$(\ol{B}_s-B_s)_{s\ge0}$, $(\ol{B}_s-B_s)_{s\le 0}$ has at the value $\ell$ of
its local time a nontrivial excursion, call it $\gep$, and moreover that
excursion makes the graph of $B$ go deeper than $-\Theta_\ell$. In Figure
\ref{fig2}, the excursion comes from $(\ol{B}_s-B_s)_{s\ge0}$. Let
$$h(\ell):=\ell+\Theta_\ell,$$ 
call $\widetilde h(\ell)$ the height of $\gep$, and $\ell^+:=\min\{x\in \C{L}:
x>\ell\}.$
$x_B$ jumps at the ``time'' $h(\ell)$, and its value just after $h(\ell)$ is
contained in the ``actual domain'' of the excursion $\gep$. The excursion may contain more
than one value of $x_B$ (e.g., in Figure \ref{fig2} it contains two, marked
with a dot). After we take into account the jumps that happen in moving through
these values, we wait until $\Theta$ jumps again at $\ell^+$ because of a new excursion that goes deeper. 

\subsection{The underlying renewal}  \label{CenterProcess}

We will now examine the distribution of the points $\{h(\ell): \ell\in \C{L}\}$.
Fix $\ell\in\C{L}$. 
For simplicity, we will denote $h(\ell), \widetilde h(\ell), h(\ell^+)$ by $h,
\widetilde h, h^+$ respectively.  
\begin{lemma}\label{overshoots} \begin{enumerate}[(i)] 
\item The random variables $\widetilde h/h, h^+/\widetilde h$ are independent
of each other and of $B \vert [H^-_{\ell},
H^+_{\ell}]$, and have density $x^{-2} 1_{x\ge1}$ and $2
x^{-3}1_{x\ge1}$ respectivelly.

\item  $$\log h^+-\log \widetilde h, \log \widetilde h-\log h$$ have exponential
distribution with means 1/2 and 1  respectively. 
\end{enumerate}
\end{lemma}

\begin{proof}

(i) Pick $\gd>0$ arbitrary. First, we prove the claim for $\ell$ being
the smallest
element of $\C{L}\cap[\gd,
\infty)$. Let (see Figure \ref{fig2b})
\begin{figure}[htbp] 
 \begin{center}
 \resizebox{10cm}{!} {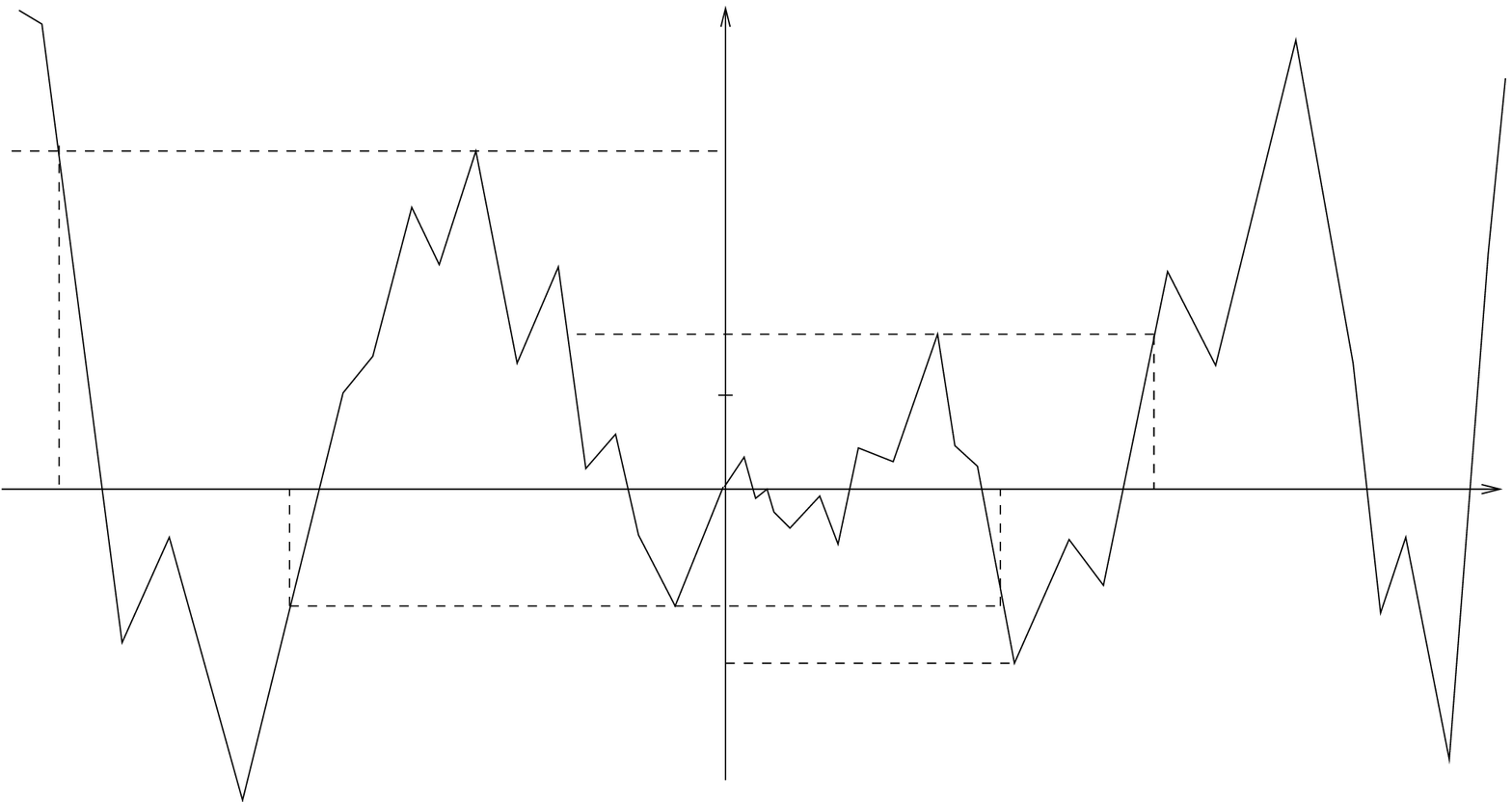} \caption{}
 \label{fig2b}
 \end{center}
 \end{figure}
\begin{align*}
\tau^+&:=\inf\{s\ge H_\gd^+: B_s=-\Theta_\gd\},\\ 
M^+&:=\ol{B}_{\tau^+},\\
\rho^+&:=\inf\{s>\tau^+: B_s=M^+\},\\
J^+&:=M^++\Theta_\gd,
\intertext{and similarly on the negative semiaxis,}
\tau^-&:=\sup\{s\le H_\gd^-: B_s=-\Theta_\gd\},\\ 
M^-&:=\ol{B}_{\tau^-},\\
\rho^-&:=\sup\{s<\tau^-: B_s=M^-\},\\
J^-&:=M^-+\Theta_\gd.
\end{align*}
Then $\ell=M^-\wedge M^+$ and 
\begin{equation} \label{tildehh}
\frac{\widetilde h}{h}=
\begin{cases}
(M^+-\ul{B}_{\rho^+})/J^+ & \text{ if } M^-\ge M^+,  \\
(M^--\ul{B}_{\rho^-})/J^- & \text{ if } M^-<M^+.
\end{cases}
\end{equation}
For $x\ge 1$, we compute 
\begin{align*}\PP\left(\frac{M^+-\ul{B}_{\rho^+}}{J^+} \ge x \, \Big| \,
J^+\right)&=\PP(\text{Brownian Motion starting from $-\Theta_\gd$ hits
$M^+-xJ^+$ before
} M^+ \, | \, J^+)\\&=\frac{M^++\Theta_\gd}{xJ^+}=\frac{1}{x}.
\end{align*}
Since $\tau^+$ is a stopping time, $\{B_{\tau^++s}-B_{\tau^+}:s\ge0\}$ is
independent of $B| [\tau^-, \tau^+]$. Thus, given $J^+$,
$(M^+-\ul{B}_{\rho^+})/J^+$ is
independent of $B| [\tau^-, \tau^+]$, and the previous computation shows that it
is independent of $J^+$ as well and has density $x^{-2}1_{x\ge1}$. Thus,
$(M^+-\ul{B}_{\rho^+})/J^+$
is
independent of $B| [\tau^-, \tau^+]$. Similarly  $(M^--\ul{B}_{\rho^-})/J^-$ has the same density, $x^{-2}1_{x\ge1}$, and is independent of $B|[\tau^-,
\tau^+]$. Since the event $M^-\ge M^+$ is in the $\sigma$-algebra generated by  $B|[\tau^-,
\tau^+]$, these observations combined with \eqref{tildehh} imply the claim of the lemma for $\wt h/h$.

We turn now to $h^+/\wt h$. Let
\begin{align*}
\hat \tau^+&:=\inf\{s\ge H_{\ell+}^+: B_s=-\Theta_{\ell+}\},\\ 
\hat M^+&:=\ol{B}_{\hat \tau^+},\\
\hat \tau^-&:=\sup\{s\le H_{\ell+}^-: B_s=-\Theta_{\ell+}\},\\ 
\hat M^-&:=\ol{B}_{\hat \tau^-}.
\end{align*}
Here $\Theta_{\ell+}$ denotes the limit of $\Theta$ at $\ell$ from the right,
and the same remark applies to $H_{\ell+}^-, H_{\ell+}^+$.
Then $\ell^+=\hat M^-\wedge \hat
M^+, \wt h=\ell+\Theta_{\ell+}, h^+=\ell^++\Theta_{\ell+}$. So that for $x\ge
1$,
$$\PP(h^+>x \wt h \, | \, \wt h)=\PP(\text{Brownian
starting from $\ell$ hits
$x \wt h-\Theta_{\ell+}$
before}-\Theta_{\ell+}\, | \, \wt
h)^2=\left(\frac{\ell+\Theta_{\ell+}}{x\wt
h}\right)^2=\frac{1}{x^2}.$$
The strong Markov property implies that, given $\wt h$, $h^+$ is independent of
$B|[H^-_{\ell+}, H^+_{\ell+}]$, and the above computation shows that $h^+/\wt h$
is indepdendent of $B|[H^-_{\ell+}, H^+_{\ell+}]$. Note that $\wt h/h$ is
determined by $B|[H^-_{\ell+}, H^+_{\ell+}]$. Thus the claim about $h^+/\wt h$
is proved.

Having proved the result for $\ell:=\min\{\C{L} \cap[\gd, \infty)\}$, we can
prove it similarly for $\ell^+$ by repeating the above procedure with the role
of $\gd$ played now by $\ell^+$. Doing the appropriate induction, we get the
result for all elements of $\C{L}\cap[\gd, \infty)$. But $\gd$ was
arbitrary, so the claim is true for all $\ell\in\C{L}$. 

\medskip

\nid (ii) It is an immediate consequence of part (i).
\end{proof}

Lemma \ref{overshoots} shows that $\{\wt h(\ell)/h(\ell),
h^+(\ell)/\wt h(\ell):\ell\in \C{L}\}$ are all independent
because for given $\ell\in\C{L}$, the ones with index strictly less than $\ell$
are functions of $B \vert [H^-_{\ell}, H^+_{\ell}]$, 
while $\wt h(\ell)/h(\ell), h^+(\ell)/\wt h(\ell)$ are independent of that path
and of each other. Also their distribution is known. Thus
$$\{\log h(\ell^+)-\log h(\ell): \ell\in \C{L}\}$$
are i.i.d. each with law the same as $W_1+W_2$, with $W_1\sim$  Exponential(1),
$W_2\sim$  Exponential(2) independent. Let
$$\psi:=\{\log h(\ell): \ell\in\C{L}\}.$$
For a given $a>0$, scaling invariance of Brownian motion implies that
$\{h(\ell):\ell\in\C{L}\}\overset{d}{=}\{a h(\ell): \ell\in \C{L}\}$. Combining
these observations, we have that $\psi$ is a stationary renewal process with
interarrival times distributed as $W_1+W_2$ mentioned above.

$x_B$ jumps at each point of $h(\C{L})$, thus $\psi\subset \xi$. In fact, the
inclusion is strict, and the points in $\xi\setminus \psi$ are the subject of
the next subsection.    

\subsection{Jumps inside an excursion. Distribution of the clusters}
\label{InsideAnExcursion}

Now we
examine the behavior of $x_B$ in each interval $[h(\ell), h(\ell^+)]$, where
$\ell\in\C{L}$. Again, we abbreviate $h(\ell), h(\ell^+)$ to $h, h^+$. Assume
that the jump at $\ell$ is caused by an excursion, $\gep$, of
$(\ol{B}_s-B_s)_{s\ge0}$. This excursion is simply 
$(B_{H_\ell^+}-B_{H_\ell^++s}: 0\le s\le
H_{\ell+}^+-{H_\ell^+})$ and contains all the information
on the jumps of $x_B$ in $[h, h^+)$.

\medskip

\textsc{Claim}: Given $h$, the excursion $\gep$ has law $n(\,\cdot\,\,
| \,
\ol{\gep}\ge h)$.

\medskip

Recall the excursion processes $\{e_t^+: t\ge0\}$ and $\{e_t^-: t\ge0\}$
introduced just after relation \eqref{Theta}. They are independent and
identically distributed, and we call $n$ their characteristic measure.
We prove the claim for $\ell:=\inf(\C{L}\cap[\gd, \infty))$, where $\gd>0$ is
arbitrary. An argument similar with the one used in the proof of Lemma
\ref{overshoots} gives the result for any $\ell\in\C{L}$. 

With probability
1, $\C{L}$ does not contain $\gd$. Let 
\begin{align*}
\tau^-&=\inf\{t\ge \gd:\ol{e_t^-}>t+\Theta_\gd\},\\
\tau^+&=\inf\{t\ge \gd:\ol{e_t^+}>t+\Theta_\gd\}.
\end{align*}
Then $\ell=\tau^-\wedge \tau^+$, and 
\begin{equation} \label{NewExcursion}
 \gep=\begin{cases}
  e^+_{\tau^+}  &\text{ if }   \tau^-\ge\tau^+,\\ 
e^-_{\tau^-}  &\text{ if }   \tau^-<\tau^+. 
     \end{cases}
\end{equation}

 The process $t\mapsto (t,
e_t^+)$ is a
Poisson point process with characteristic measure $\gl \times  n$ ($\gl$ is
Lebesgue measure), and $\tau^+$
is the first entrance time of this process  in the set $A:=\{(s, \gep): s\ge\gd,
\ol{\gep}>s+\Theta_\gd\}$. The law of the pair $(\tau^+,
e^+_{\tau^+})$ is
that
of  $\gl \times  n(\, \cdot \, | \, (s, \gep)\in A)$, and given that
$\tau^++\Theta_\gd=h$, the law of $e^+_{\tau^+}$ is
independent of $\tau^+$ and  equals 
$n(\,\cdot\,\,|
\,
\ol{\gep}>h)$, which the same as $n(\,\cdot\,\, | \,
\ol{\gep}\ge h)$. The analogous assertion holds for the pair
$(\tau^-, e^-_{\tau^-})$, which is independent of $(\tau^+,
e^+_{\tau^+})$. These observations together with
\eqref{NewExcursion} imply the  claim.

\medskip

We pause for a moment to define for any excursion $\gep_0$ of $\ol{B}-B$ and
$a>0$, a positive integer $\C{N}(\gep_0, a)$. 

\nid Assume that $\gep_0$ has domain $[0, \gz]$ and height $\widetilde h:=\ol{\gep_0}>0$.
We consider the path $\gga=-\gep_0$, see Figure \ref{fig3}. To the process
$(\gga_s-\ul{\gga}_s)_{s\in [0,\gz]}$ corresponds the process 
$(\gep_r)_{r\in[0,\widetilde h]}$ of its excursions away from zero. This is parametrized by the inverse of the local time process defined by
the absolute value of the running minimum (i.e., -$\ul{\gga}_s$).
Since $\gep_0$ is continuous defined on a compact interval, the subset of
excursions with height $\ge a$ constitutes a finite, possibly empty, set
$(\gep_{r_i})_{1\le i\le K}$, with
$(r_i)$ increasing. We define recursively a finite sequence $j$ as
follows (see Figure \ref{fig3}).
\begin{align*}
 j_0:&=0\\
 j_1:&=\min\{i\le K: \ol{\gep_{r_i}} >a\} \text{ \ \ \ \ if the set is
nonempty, }  \\
j_{k+1}:&=\min\{i\le K: \ol{\gep_{r_i}} > \ol{\gep_{r_{j_k}}}\}
\text{\ \ \ if the set is nonempty and }k\ge 1.
\end{align*}
 If any of the sets involved in the definition
is empty, the corresponding $j_k$ is not defined, and the
recursive definition stops. Let
$$\C{N}(\gep_0, a):=
\begin{cases}\max\{k:j_k \text{ is defined}\}+1 & \text{ if } \ol{\gep}\ge
a,\\
0 &\text{ if } \ol{\gep}<a.
\end{cases}
$$
Recalling the definition of $x_B$, we can say informally that $\C{N}(\gep_0, a)$
counts the number of jumps caused in $x_B$ by $\gep_0$ with starting benchmark
$a$. 

\begin{figure}[htbp] 
 \begin{center}
 \resizebox{9cm}{!} {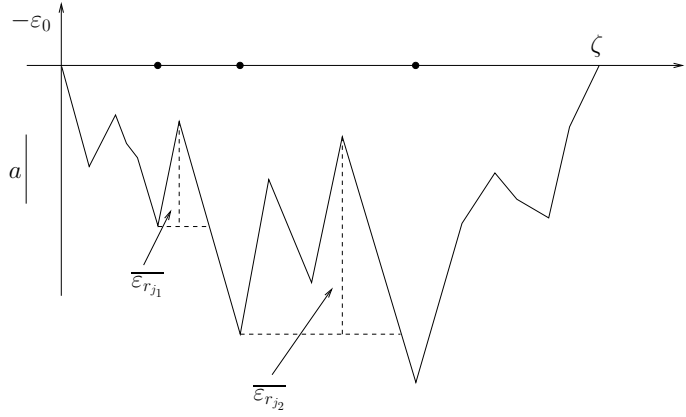} \caption{The graph of $-\gep_0$. For
this path, only $j_0, j_1, j_2$ are defined, thus $\C{N}(\gep_0, a)=3$. The
three points on the $x$-axis mark the values of $x_B$ after $x_B(a)$ that are
contained in the ``actual domain'' of the excursion.}
 \label{fig3}
 \end{center}
 \end{figure}

Thus $\C{N}(\gep, h)$ counts the jumps of $x_B$ in $[h, h^+)$, and note that
$\C{N}(\gep, h)\ge 1$ because of the jump at $h$, while in the interval 
$[\widetilde h, h^+)$ there are no jumps.  
The excursions $\gep_{r_{j_k}}$ in the definition of $\C{N}(\gep, h)$ give rise
to the jumps in
$(h, \widetilde h)$. 
And in fact, if we let $\nu:=\C{N}(\gep, h)$, the jumps happen exactly at the
points
$$\ol{\gep_{r_{j_1}}}<\ldots<\ol{\gep_{r_{j_{\nu-1}}}},$$
assuming that $\nu>1$. Otherwise, there are no jumps in $(h, \widetilde h)$.

We will determine the law of these points given the value of $h$.

\nid Let $a=h$ .
The law of $-\gep$ is described as follows  (see \citet{RY}, Chapter XII,
Theorem
4.1). It starts from zero as
the negative of a three dimensional Bessel process until it hits
$-a$. After that, it continues as Brownian motion until hitting 0. Thus, let $\eta_0:=a, y_0=-a,$ and take $W$ a Brownian motion starting from $y_0$. Then let 
\begin{align*}
\tau_0&:=\min\{s>0: W_s-\ul{W}_s=\eta_0\}, \\
y_1&:=\ul{W}_{\tau_0},\\
\gs_1&:=\min\{s>\tau_0:W_s=y_1\},\\
\eta_1&:=\ol{W}_{\gs_1}-y_1.
\end{align*}
By well known property of Brownian motion, it holds $\eta_1>\eta_0$. Repeat the above, with the role of $\eta_0, y_0$ played by $\eta_1,
y_1$, and define $\tau_1, y_2, \gs_2, \eta_2$. Continue recursively. Then $\nu$
is the largest integer $i$ for which $\eta_{i-1}\le |y_{i-1}|$, while 
$$\ol{\gep_{r_{j_1}}}=\eta_1, \ol{\gep_{r_{j_2}}}=\eta_2, \ldots,
\ol{\gep_{r_{j_{\nu-1}}}}=\eta_{\nu-1},$$
and $-y_\nu=\widetilde h$, which is the height of the excursion.

We remark that given $y_k$ and $\eta_k$, the random variables $y_{k+1},
\eta_{k+1}$
are independent of $W \vert [0, \gs_k]$ because by the strong Markov property,
$W^{(k)}_s:=W_{\gs_k+s}-y_k$ is a standard Brownian motion independent of $W
\vert [0, \gs_k]$, and $y_{k+1}, \eta_{k+1}$ are functions of the path
$W^{(k)}$ and of $y_k, \eta_k$. The dependence on $y_k, \eta_k$ is removed if
we consider
$$\frac{y_k-y_{k+1}}{\eta_k}=:\ga_k,
\,\frac{\eta_{k+1}-\eta_k}{\eta_k}=:\gb_k.$$
\textsc{Claim}: The random variables $\ga_k, \gb_k$ are independent of $W \vert [0, \gs_k]$, independent of each
other, and have
densities $e^{-x} \ 1_{x>0}, (1+x)^{-2}\ 1_{x>0}$ respectively.

Indeed, consider the excursion process, for the excursions away from zero, of the reflected from the past minimum
process
$W^{(k)}-\ul{W}^{(k)}$ parametrized by the inverse of the local time process 
$L_s:=|\ul{W}^{(k)}_s|$. $y_k-y_{k+1}$ is the value of the local time when the
first excursion with height at least $\eta_k$ appears, while $\eta_{k+1}$ is
the height of the excursion. Now Proposition 2 from Chapter 0
of \citet{BE} gives that, conditional on $\eta_k$, $y_k-y_{k+1}$ is an
exponential
random variable with parameter $n(\gep\ge 
\eta_k)=1/\eta_k$, and the excursion is
independent of $y_k-y_{k+1}$ and has law $n(\,\cdot \, | \,
\ol{\gep}\ge \eta_k)$. 
The equality $n(\gep\ge 
\eta_k)=1/\eta_k$ is true by Exercise 2.10 (1), Chapter XII of \citet{RY}, which also implies that $\eta_{k+1}$
has density $\eta_kx^{-2} 1_{x\ge
\eta_k}$. So that the conditional law of $(\ga_k, \gb_k)$
given $\eta_k$ does not depend on $\eta_k$ or $y_k$ and it is a product measure.
Thus
$\ga_k, \gb_k$ do not depend on $\eta_k$ or $y_k$, are independent of each other, and
have the required density. The proof of the claim is concluded by also taking into account the discussion preceding it.
 
The above imply that $\{(\ga_k, \gb_k): k\ge0\}$ are
i.i.d.

\nid Then the random variables
\begin{equation} \label{excursionVariables}\nu, \left(\frac{|y_{k-1}|}{a},\, \log 
\frac{\eta_{k-1}}{a}\right)_{1\le k \le \nu}, \log \frac{|y_\nu|}{a}
\end{equation}
are related in exactly the same way as  
\begin{equation} \label{introductionVariables}
 \C{N}, (z_k, t_k)_{1\le k\le\C{N}}, F,
\end{equation}
defined in Section \ref{structureSSection}. In particular, they don't depend on $a$.  For example 
$$\log  \frac{\eta_{k+1}}{a}-\log  \frac{\eta_k}{a}=\log 
\frac{\eta_{k+1}}{\eta_k}=\log(1+\gb_k)$$
has exponential distribution with mean 1. The correspondence between \eqref{excursionVariables},
\eqref{introductionVariables} proves that $F$ has exponential
distribution with mean 1 because $\log (|y_\nu|/a)=\log(\widetilde h/h)$,
whose distribution was determined in Lemma \ref{overshoots}. 
Thus 
$$\xi\cap[\log h(\ell), \log h(\ell^+))\,\big|\,\ell \in \C{L}
\overset{d}{=}\log h(\ell)+\C{T}\, \big|\, \ell \in \C{L}.$$
Taking into account the structure of the process $\psi:=\{\log h(\ell):\ell\in
\C{L}\}$ given at the end of Section \ref{CenterProcess}, we get Theorem
\ref{ClusterForm}.

\section{Jumps inside an excursion. Proof of Theorem \ref{MGFProposition} }
\label{GFComputation}

For $\gl, \mu\in\D{R}$, let
\begin{equation}\label{MGFPair} K(\gl, \mu):=\EE(e^{\gl \C{N}+\mu F,
})\end{equation}
the moment generating function of $(\C{N}, F)$. The aim of this section is to
compute $K$ explicitly for all $\gl, \mu$ for which it is finite. First we
compute it for negative $\gl, \mu$, and then we use analytic extension.

\subsection{The Laplace transform} \label{LaplaceSubsection}

Recall that we denote by $\Psi$ the
confluent hypergeometric function of the second kind.
\begin{proposition}\label{LaplTrans}
It holds
\begin{equation}\label{LaplTransEqn} \EE(e^{-\gl \C{N}-\mu F })=
e^{-\gl}\frac{\Psi(1-e^{-\gl},
1-\mu;1)}{\Psi(-e^{-\gl}, -\mu;1)}
\end{equation}
for all $\gl, \mu\ge0$ .
\end{proposition}

\begin{proof}

Because of the correspondence between \eqref{excursionVariables} and 
\eqref{introductionVariables}, the pair $(\C{N}, F)$ has the same distribution
as $(\C{N}(\gep, 1), \log \ol{\gep})$ where $\gep$ is an excursion
with law $n(\cdot\, | \,
\ol{\gep}\ge 1)$. In the following, we use the notation
set in Section \ref{InsideAnExcursion}, with $h=1$, and in particular the random
variables $\{y_k, \eta_k, \ga_k, \gb_k: k\ge0\}$. Let
\begin{align} (s_k, x_k)&:=\left(\frac{|y_k|}{\eta_k}, |y_k|\right) \text{ and
} 
\phi_k=(1+\gb_k)^{-1} \\ 
\intertext{for $k\ge0$. Then, $(s_0, x_0)=(1,1)$,}
 \label{CEvolution1}
s_{k+1}&=\frac{x_k+\ga_k \eta_k}{\eta_k
(1+\gb_k)}=(s_k+\ga_k)\phi_k, \\
x_{k+1}&=x_k\left(1+\frac{\ga_k}{s_k}\right), \label{CEvolution2}
\end{align}
for all $k\ge0$, while by the claim in Section \ref{InsideAnExcursion},
$\{\ga_k, \phi_k: k\ge0\}$ are independent, with $\ga_k$ exponential with
mean 1 and $\phi_k$ uniform in $[0,1]$. Also  
\begin{align}\C{N}(\gep, 1)&=\min\{i>1:s_i<1\},\\
\ol{\gep}&=x_{\C{N}(\gep, 1)}.
\end{align}

\nid  Fix $\gl, \mu\ge0$. For given $(s, x)\in [1, \infty)\times (0, \infty)$,  
consider the Markov process $(s_k, x_k)_{k\ge0}$ that has $(s_0, x_0)=(s, x)$ and evolves 
as in \eqref{CEvolution1}, \eqref{CEvolution2}, and define 
\begin{equation}
\begin{aligned}
M&=M(s, x):=\min\{i>1:s_i<1\},\\
f(s,x)&:=\EE_{s_0=s, x_0=x}(e^{-\gl M-\mu\log x_M}\textbf{1}_{M<\infty})=\EE_{s_0=s,
x_0=x}(e^{-\gl M} x_M^{-\mu} \textbf{1}_{M<\infty}).
\end{aligned}
\end{equation}
We will show that $M<\infty$ with probability 1, so that $K(-\gl, -\mu)=f(1,1)$. Thus the plan is to show that $f$ is regular enough, derive a differential equation
involving it, and solve the equation to get in particular the value $f(1,1)$.

Using standard arguments, we can see that $f$ is measurable. Also, it
is nonnegative and bounded by $\gd^{-\mu}$ in each set of the form
$[1, \infty)\times[\gd, \infty)$, with $\gd>0$, because $\gl, \mu,
M\ge0$, and by \eqref{CEvolution2}, $(x_k)_{k\ge0}$ is increasing. 

Brownian scaling gives that
\begin{equation}\label{scaling}
f(s,x)=x^{-\mu}f(s,1).
\end{equation}
For $(s, x)\in [1, \infty)\times (0, \infty)$, define
\begin{equation}
H(s,x):=\int_1^s f(t,x)\, dt+x^{-\mu}.
\end{equation}
\medskip

\nid \textsc{Claim:} It holds
\begin{equation}\label{PDE} s\, \partial_{s,s } H(s,x)+x\, \partial_{s, x} 
H(s,x)-s\, \partial_s H(s,x)+e^{-\gl}H(s,x)=0
\end{equation}
in the interior of
$$\{(s,x): s\ge 1, x>0\},$$
and $H(1,x)=x^{-\mu}$ for $x>0$.

\medskip

\textsc{Proof of the claim:} The equation is derived through first step
analysis. Call $k(dt, dy \vert s, x)$ the transition law of the chain $(s_n,
x_n)_{n\ge1}$. Then using \eqref{CEvolution1}, \eqref{CEvolution2} we have that
\begin{equation}\label{integrEquation} f(s,x)=e^{-\gl}\left(\EE(x_1^{-\mu}
1_{s_1<1})+\int_{A_{s,x}} f(t,y) k(dt, dy \vert s, x)\right),
\end{equation}
with
$$A_{s,x}:=\left\{(t,y): 1\le t \le \frac{s}{x} y, y\ge x\right\}.$$
For fixed $s, x$, the measure $k(dt, dy \vert s, x)$ is supported on $$B_{s,
x}:=\left\{(t,y): 0<t \le \frac{s}{x} y, y\ge x\right\}$$ 
and is derived from a density, which we now determine. The distribution
function of the measure at a $(t, y)\in B_{s, x}$ is 
\begin{equation}
\begin{aligned}F(t,y)&:=\PP\left((s+r) \phi\le t,
x\left(1+\frac{r}{s}\right)\le
y\right)=\PP\left(r\le \left(\frac{y}{x}-1\right)s, \phi\le
\frac{t}{s+r}\right)\\&=\int_0^{(\frac{y}{x}-1)s}e^{-z}
\left(\frac{t}{s+z}\wedge 1\right)\, dz=
\begin{cases}
t\int_0^{(\frac{y}{x}-1)s}\frac{e^{-z}}{z+s}\, dz & 0<t\le s, \\
\int_0^{t-s} e^{-z} \, dz+t\int_{t-s}^{(\frac{y}{x}-1)s}\frac{e^{-z}}{z+s}\, dz
& s<t\le \frac{y}{x} s.
\end{cases}
\end{aligned}
\end{equation}
In the interior of $B_{s, x}$, $\partial_t F(t, y)$ exists and is continuous in
$t$, and $\partial_{y, t} F(t, y)$ exists and is continuous in $y$. Also, the
integral of $\partial_{y, t} F(t, y)$ in $B_{s, x}$ is 1. Thus, the measure
$k(dt, dy \vert s, x)$ has density
$$\frac{\partial^2F}{\partial y \partial t}(t, y)\textbf{1}_{(t, y)\in B_{s, x}}
=\frac{1}{y}e^{-(\frac{y}{x}-1)s} \textbf{1}_{(t, y)\in B_{s, x}}.$$
Let $g(y)=y^{-\mu}$. Then \eqref{integrEquation} becomes
\begin{equation}
f(s,x)=e^{-\gl}\left(\int_x^{\infty} \frac{g(y)}{y}
e^{-(\frac{y}{x}-1)s}
dy+\int_x^\infty\int_1^{\frac{s}{x} y}\frac{1}{y}e^{-(\frac{y}{x}-1)s} f(t,y) \,
dt\, dy\right).
\end{equation}
This, combined with the measurability and boundedness of $f$ in
sets of the form $[1, \infty)\times[\gd, \infty)$, with $\gd>0$,
shows that $f$ is
continuous in $[1, \infty)\times(0, \infty)$ and differentiable in the interior
of the same set. We write the last equation as 
\begin{align*}e^{\gl-s} f(s,x)&=\int_x^{\infty} \frac{g(y)}{y} e^{-\frac{y}{x}s}
dy+\int_x^\infty\int_1^{\frac{s}{x} y}\frac{1}{y}e^{-\frac{y}{x}s} f(t,y) \,
dt\, dy\\
&=\int_s^{\infty} \frac{e^{-w}}{w} g\left(\frac{x}{s} w\right) dw+\int_s^\infty
\frac{e^{-w}}{w}   \int_1^w  f\left(t,
\frac{x}{s} w\right) \, dt\, dw.
\end{align*}
Putting  $x=hs$ we get
$$e^{\gl-s} f(s,h s)=
\int_s^{\infty} \frac{e^{-w}}{w} g\left(h w\right) dw+\int_s^\infty
\frac{e^{-w}}{w}   \int_1^w  f\left(t, h w\right) \, dt\, dw,$$
and differentiating with respect to $s$, 
$$e^{\gl-s}\big(-f(s,h s)+\partial_s f(s,hs)+h\partial_x
f(s,sh)\big)=-\frac{e^{-s}}{s}\, g(hs) -\frac{1}{s}e^{-s} \int_1^{s} f(t,h s) \,
dt.$$
Here $\partial_s, \partial_x$ denote differentiation with respect to the first and second argument respectively. Putting back $h=x/s$, this gives
$$e^{\gl}\big(-f(s,x)+\partial_s f(s,x)+\frac{x}{s}\, \partial_x
f(s,x)\big)+\frac{1}{s}g(x)+\frac{1}{s} \int_1^{s} f(t,x) \, dt=0,$$
which in terms of $H(s,x)$ is written as
$$s \big(-\partial_s H(s,x)+\partial_{s s} H(s,x)\big)+x\partial_{s x} 
H(s,x)+e^{-\gl}H(s,x)=0.$$
This is \eqref{PDE}.

\medskip

\nid \B{Determination of $f$.}

\medskip

For $s\ge1$ define $G(s):=H(s, 1)$.  Relation \eqref{scaling} gives
$H(s,x)=x^{-\mu} G(s)$, so that \eqref{PDE} is equivalent to
\begin{equation}\label{DYE}
 s\, G''(s)+(-\mu-s)\, G'(s)+e^{-\gl}\, G(s)=0,
\end{equation}
while the condition $H(1,x)=x^{-\mu}$ translates to $G(1)=1$.

\nid Let $a:=-e^{-\gl}$. For $\mu\notin \D{N}=\{0, 1, \ldots\}$, the general
solution of
\eqref{DYE} is (see (9.10.11) of \citet{LE})
$$C_1\Phi(a, -\mu; s)+C_2\Psi(a, -\mu; s)$$
with $\Phi, \Psi$ the confluent hypergeometric functions of the first and second kind respectively.

Restrict first to the case $\mu>0, \mu\notin \D{N}, \gl>0$. Then as
$s\to\infty$, $|\Phi(a, -\mu; s)|$ goes to infinity faster than any
polynomial (see relation (9.12.8) of \citet{LE}),
while $\Psi(a, -\mu; s)/s\to0$ because of \eqref{twominusb},
\eqref{PhiAsymptotic}, and noting that $a\in(-1, 0)$.
Since $|G(s)|\le s$, we get $C_1=0$. Then
$G(1)=1$ gives that
\begin{equation}
G(s)=\frac{\Psi(a, -\mu; s)}{\Psi(a, -\mu; 1)}.
\end{equation}
Note that the denominator is not zero because by \eqref{twominusb} it equals $\Psi(a+\mu+1,
2+\mu;1)$, which, because of \eqref{integrrepr}, is positive.

\nid Then
\begin{equation} f(s,x)=x^{-\mu}f(s,1)=\partial_s H(s,x)=x^{-\mu}
G'(s)=x^{-\mu}\frac{(-a)\Psi(a+1,1-\mu;s)}{\Psi(a,-\mu;1)}
\end{equation}
because of \eqref{PsiDerivative}, that is
\begin{equation} \label{PreMGF}
\EE_{s_0=s, x_0=x}(e^{-\gl M}
x_M^{-\mu}\textbf{1}_{M<\infty})=e^{-\gl}\frac{\Psi(1-e^{-\gl},1-\mu;s)}{\Psi(-e^{-\gl},
-\mu;1)}.
\end{equation}
The quantity in the expectation, for $\mu\in[0, 1], \gl\ge0$,  is
bounded by $\max\{1, x^{-1}\}$ because by \eqref{CEvolution2},
$(x_k)_{k\ge0}$ is increasing, 
and thus when sending $\gl, \mu\to 0^+$ in the last equality, we can
invoke the bounded convergence theorem to get $\PP_{s_0=s,
x_0=x}(M<\infty)=1$. We used \eqref{zero}, \eqref{minusone} for the
evaluation of the right hand side of the equality. 

Now using the continuity of both sides of \eqref{PreMGF} in $\mu$, we infer its validity for
$\mu\in\D{N}$ too. And similarly for $\mu\ge 0$ and $\gl=0$. In particular,
\begin{equation}\EE(e^{-\gl \C{N}-\mu F})=f(1,1)=e^{-\gl}\frac{\Psi(1-e^{-\gl},1-\mu;1)}{\Psi(-e^{-\gl},
-\mu;1)}
\end{equation}
for all $\gl, \mu\ge0$.
\end{proof}

\subsection{Analytic extension}
Our objective in this subsection is to extend equality
\eqref{LaplTransEqn} to all values of $\gl, \mu$ for which the left
hand side is finite.
Before proceeding, we collect some
facts concerning the function $\Psi$ which we will use in the rest of
the paper. For their proof, we refer the reader to \citet{LE}. 

$\Psi(\,\cdot\,, \,\cdot\, ;
\,\cdot\,)$
is defined in $\D{C}\times\D{C}\times(\D{C}\setminus (-\infty, 0])$ and is
analytic in all its arguments (\S 9.10 of \citet{LE}). Differentiation with respect to the first, second, and third argument will be denoted
by
$\partial_x, \partial_y, \partial_z$ respectively.
 In its domain, $\Psi$ 
satisfies
\begin{align}
\Psi(a, b; z)&=z^{1-b}\Psi(a-b+1, 2-b; z), \label{twominusb} \\
\Psi(a-1,b;z)+(b-2a-z)&\Psi(a,b;z)+a(a-b+1)\Psi(a+1,b;z)=0,
\label{arec}\\
\Psi(a-1,b;z)-z\Psi(a,b+1;z)&=(a-b)\Psi(a,b;z),
\label{abrec}\\
\partial_z \Psi(a, b; z)&=-a\Psi(a+1, b+1; z), \label{PsiDerivative}
\end{align}
while for $a,z$ with positive real part, it holds
\begin{equation}\label{integrrepr}
\Psi(a,b;z)=\Gamma(a)^{-1}\int_0^{+\infty} e^{-z t}
t^{a-1}(1+t)^{-a+b-1}dt.
\end{equation}
Relations \eqref{twominusb}, \eqref{arec}, \eqref{abrec}, \eqref{PsiDerivative},
\eqref{integrrepr} are respectivelly (9.10.8), (9.10.17), (9.10.14), (9.10.12), (9.11.6) 
of \citet{LE}.

We will also need some special values of $\Psi$
\begin{lemma} For $a\in(0, \infty), b\in \D{C}, z\in \D{C}\setminus(-\infty, 0]$, it holds
\begin{align}
\Psi(0, b; z)&=1, \label{zero} \\
\Psi(-1, b;z)&=z-b, \label{minusone}\\
\lim_{z\to+\infty} z^a \Psi(a, b; z)&=1, \label{PhiAsymptotic}
\intertext{while}
\partial_x\Psi(0, 1; 1)&=0, \label{Derivative01}\\
\partial_x\Psi(-1, 0; 1)&=1, \label{Derivative10}\\
\partial_x\Psi(0, 0; 1)&=-\int_0^\infty e^{-t} (1+t)^{-1}\, dt. \label{Derivative00}
\end{align}
\end{lemma}

\begin{proof}
For \eqref{zero}, note that by
\eqref{PsiDerivative}, $\Psi(0, b; z)$ is a function of $b$ alone, while it is
easy to see that for $z>0$, $\lim_{a\to 0^+} \Psi(a, b;z)=1$ (use
\eqref{integrrepr} and $\lim_{a\to 0^+} a \hspace{0.2ex} \Gamma(a)=1$). 
Then \eqref{minusone} follows from \eqref{zero} and \eqref{arec} by setting $a=0$.

Relation 
\eqref{PhiAsymptotic} follows from \eqref{integrrepr} by doing the change
of variables $y=zt$ in the integral and applying the dominated convergence
theorem.

\nid Regarding \eqref{Derivative01}, note that $\Psi(0, 1;1)=1$, and for $x>0$,
$$\frac{\Psi(x, 1;1)-1}{x}=\frac{1}{x\Gamma(x)}\int_0^\infty e^{-t} t^{x-1}\big\{(1+t)^{-x}-1\big\}\, dt.$$
For $x\to0^+$, the denominator goes to 1, while the numerator goes to zero by the dominated convergence theorem.

\nid Finally, \eqref{Derivative10} follows from\eqref{abrec}, \eqref{zero} and \eqref{Derivative01},  while \eqref{Derivative00} is proven in the same way as \eqref{Derivative01} taking into account that $\Psi(0, 0;1)=1$.
\end{proof}

Define
\begin{equation}\Xi(\gl, \mu):=e^{\gl}\frac{\Psi(1-e^{\gl},
1+\mu;1)}{\Psi(-e^{\gl}, \mu;1)}
\end{equation}
for all $\gl, \mu\in \D{C}$ that this makes sense, that is, everywhere except possibly at values where the denominator is 0. Proposition \ref{LaplTrans}
shows that $\Xi(\gl, \mu)=\EE(e^{\gl \C{N}+\mu F})=:K(\gl, \mu)$ for $\gl,
\mu\le 0$. We show below that this holds throughout 
$$D_K:=\{(\gl, \mu)\in\D{R}^2: K(\gl, \mu)<\infty\}.$$
The following two lemmas
show, among other things, that $D_K$ contains a neighborhood of $(0,
0)$.

\begin{lemma}\label{FiniteGF}
The number
$$z_0:=\sup\{z>0: \EE(z^{\C{N}})<\infty\}\in(1,2),$$
and $\EE(z_0^{\C{N}})=\infty$.
\end{lemma}

\nid Mathematica gives the
approximate value $z_0\approx 1.57391$

\begin{proof}
By Proposition \ref{LaplTrans}, we have
\begin{equation}\label{GF}\EE(z^{\C{N}})= z\frac{\Psi(1-z,
1;1)}{\Psi(-z, 0;1)}
\end{equation}
for all $z$ with $z\in[0, 1]$. Call $W(z)$ the right hand side of \eqref{GF}.
The left hand side is a power series in $z$ with positive coefficients
$a_k:=\PP(\C{N}=k)$ for all $k\ge0$. The right hand side is a meromorphic
function on the plane. It is finite at 0, so that it has a power series
development centered at zero. Since the coefficients are positive, the radius of
convergence coincides with the smallest pole of $W$ on $(0, \infty)$. We will
show that this occurs at a point  $z_0\in(1,2)$.

The denominator in \eqref{GF} is a continuous function of $z$ and equals $\Psi(1-z, 2,1)$, which is positive in
$[0,1]$ and has value -1 at $z=2$ (use \eqref{twominusb},
\eqref{integrrepr}, \eqref{minusone} correspondingly for the last three claims). Thus, 
it has a smallest root
in $(1,2)$, call it $z_0$. On the other hand, the numerator is positive in $[0, 2)$. To see that,
let $y=2-z$,  and note that, since $y>0$, \eqref{arec} and  the integral representation 
\eqref{integrrepr} give 
$$\Psi(y-1, 1,1)=y
(2\Psi(y,1,1)-y \Psi(y+1,1,1))=y\Gamma(y)^{-1}\int_0^{+\infty} e^{-t}
t^{y-1}(1+t)^{-y-1}(t+2)\, dt>0.$$ 
Thus, the power series
for $W(z)$ centered at 0 has radius of convergence $z_0$. As we already noted,
the
expectation on the left hand side of \eqref{GF} is a power
series of $z$. It follows that it too has radius of convergence
$z_0$, thus the two sides of \eqref{GF} are finite and equal for all $z\in\D{C}$ with $|z|<z_0$.
The fact that $z_0$ is a pole of $W$ gives that $\EE(z_0^\C{N})=\infty$ and concludes the proof of the lemma.  \end{proof}

Next, we list some properties of the set $D_K$.

\begin{lemma}

\begin{enumerate}

\item $D_K$ is convex. \label{Pconvex}

\item $(x, y)\in D_K$ implies that $(-\infty, x]\times(-\infty, y]\in D_K$. \label{PclosedWS}

\item $(\gl, 0)\in D_K$ exactly when $\gl<\gl_0:=\log z_0>0$. \label{PxAxis}

\item $(0, \mu)\in D_K$ exactly when $\mu<1$. \label{PyAxis}

 \item The intersection of $D_K$ with the second and fourth quadrant is under
the line that passes through $(\gl_0, 0), (0, 1)$. \label{PUpperLine}

\item The interior of the triangle with vertices $(0,0), (0, 1), (\gl_0, 0)$ is
inside $D_K$. \label{PTriangle}

\end{enumerate}

\end{lemma}

\begin{proof}

\ref{Pconvex} follows from H\"older's inequality, \ref{PclosedWS} is true because $\C{N}$ and $F$ take positive values, \ref{PxAxis} is shown in Lemma
\ref{FiniteGF}, \ref{PyAxis} follows from the fact that $F\sim$ Exponential(1), and finaly \ref{PUpperLine} and \ref{PTriangle} follow from \ref{Pconvex},  \ref{PclosedWS}, \ref{PxAxis}, \ref{PyAxis}.
\end{proof}

And now we are ready to state the main result of this subsection,
which completes the proof of Theorem \ref{MGFProposition}.

\begin{lemma} \label{AnalyticExtension}

1. $K(\gl, \mu)=\Xi(\gl, \mu)$ for every $(\gl, \mu)\in D_K$.

2. $K(\gl, \mu)=\infty$ for every $(\gl, \mu)\in \partial D_K$. In particular,
$D_K$ is open.

3. $\EE(\C{N})=2$.
\end{lemma}

\begin{proof}

\nid 1 and 2. Fix $\mu\le 0$. Since $\EE(e^{\gl \C{N}+\mu F})$ is finite for $\gl\in[0,
\gl_0)$, it follows that the power series in $\gl$
$$\EE(e^{\gl \C{N}+\mu F})=\sum_{k=0}^\infty \frac{1}{k!}\EE(e^{\mu F}
\C{N}^k)\gl^k $$
has radius of convergence at least $\gl_0$. Also
$$\gl\mapsto e^{\gl}\frac{\Psi(1-e^{\gl},
1+\mu;1)}{\Psi(-e^{\gl}, \mu;1)},$$
is analytic near zero because the value of the denominator at 0 is $1-\mu\ne0$ (recall \eqref{minusone}), and $\Psi$ is entire in its first argument. Since it agrees with the previous power series in a line
segment, they agree on the ball of convergence of the series. In particular, its
development around zero has positive coefficients and consequently its radius of
convergence, $\hat \gl(\mu)$, coincides with its smallest singularity in $[0, \infty)$ if such exists, otherwise it is infinite. Since the numerator is entire in $\gl$, the only possibility for a singularity is at a zero of the denominator. Thus $K(\gl, \mu)<\infty$ exactly for $\gl<\hat \gl(\mu)$ and for all such
$\gl$ it holds $K(\gl, \mu)=\Xi(\gl, \mu)$. Because of Property \ref{PUpperLine} of the previous lemma, it
follows that $\gl(\mu)<\infty$. Property 1 gives that $\mu\mapsto \hat \gl(\mu)$ is
concave in $(-\infty, 0]$, thus continuous in $(-\infty, 0)$, and Properties \ref{Pconvex},  \ref{PclosedWS}, \ref{PxAxis}  give that it is also left continuous at zero with value $\hat \gl(0)=\gl_0$.

Now fix $\gl<\gl_0$. $\EE(e^{\gl \C{N}+\mu F})$ is finite for small enough
positive $\mu$ due to Property \ref{PTriangle}. With similar reasoning as
above, we show that there is a concave function $\gl\mapsto \hat\mu(\gl)$ continuous
on $(-\infty, \gl_0]$, $\hat \mu(\gl_0)=0$, so that for $\gl\in(-\infty, \gl_0]$ it
holds $K(\gl, \mu)<\infty$ iff $\mu<\hat \mu(\gl)$ and moreover $K(\gl, \mu)=\Xi(\gl,
\mu)$. 

Thus $$\partial D_K=\{(\hat\gl(\mu), \mu): \mu\le 0\}\cup \{(\gl, \hat\mu(\gl)): \gl\le\gl_0\},$$
and on this set $K$ takes the value $\infty$.
This finishes the proof of the first two statements.

3. It follows from the first claim of the Lemma, the formula for $\Xi$, and
differentiation.
\end{proof}

\section{Proof of Theorem \ref{fluctuations}} \label{ProofCLT}

Let $(S_k)_{k\in\mathbb{Z}}$ be the points of the renewal $\psi$ in increasing
order such that $S_{-1}<0\le S_0$, and for $k\in\mathbb{Z}$,
\begin{align*}
X_k&:=S_k-S_{k-1},\\
\C{N}_k&:=N[S_{k-1}, S_k).
\end{align*}
The random variables $\{(X_k, \C{N}_k): k\ge 1\}$ are i.i.d.,  
each with distribution the same as $(F+Z, \C{N})$, defined in Section \ref{structureSSection}. Then $\EE(X_1)=1+(1/2)=3/2$, and $\EE(\C{N})=2$ by Lemma \ref{AnalyticExtension}. Let also for $k\ge1$, 
$$Y_k:=\C{N}_k-a X_k,$$ 
where $a:=\EE(\C{N}_1)/\EE(X_1)=4/3$. Then $\{Y_k:k\ge1\}$ are i.i.d. with mean value 0, and we will see below that they have finite variance. By the central limit theorem,
\begin{align*}\frac{N[0, S_k)-a S_k}{\sqrt{k}}&=\frac{N[0,
S_0)+\C{N}_1+\cdots+\C{N}_k-a(S_0+X_1+\cdots+X_k)}{\sqrt{k}}\\
&=\frac{Y_1+\cdots+Y_k+N[0, S_1)-aS_0}{\sqrt{k}}
\Rightarrow  N(0, \var(Y_1))
\end{align*}
for $k\to\infty$. 

\nid For $t>0$, let $n_t:=\max\{k: S_k\le t\}$. Then, by the renewal theorem, we have $\lim_{t\to\infty} n_t/t=1/\mu$ with
$\mu:=\EE(X_1)=3/2$, thus in the same way as in Exercise 3.4.6 in \citet{DU}, we
get 
\begin{equation}\label{CLT1}\frac{N[0, S_{n_t})-a S_{n_t}}{\sqrt{t}}\Rightarrow 
N(0, \var(Y_1)/\mu).
\end{equation}
Now note that the families $\{S_{n_t}-t: t>0\}, \{N[0, t]-N[0, S_{n_t}): t>0\}$
are tight, because by stationarity, for every $t>0$,
\begin{align*}0&\le t-S_{n_t}\le S_{n_t+1}-S_{n_t}\overset{d}{=} 
S_0-S_{-1},\\
0&\le N[0, t]-N[0, S_{n_t})\le N[S_{n_t}, S_{n_t+1})\overset{d}{=} N[S_{-1},
S_0).
\end{align*}
Thus \eqref{CLT1} and Slutsky's theorem give 
\begin{equation}\label{CLT2}\frac{N[0, t]-at}{\sqrt{t}}\Rightarrow  \mathcal{N}(0,
\var(Y_1)/\mu).
\end{equation}
It remains to compute $\var(Y_1)$.  We have  $Y_1\overset{d}{=}\C{N}-a(F+Z)$, and recall that $F\sim$ Exponential(1), $Z\sim$ Exponential(2) is independent of $(\C{N}, F)$, and the moment generating function of $(\C{N}, F)$ is given in Theorem \ref{MGFProposition}. Thus
\begin{align}\label{VarComputation} \var(Y_1)
&=\var(\C{N})+a^2\var(F+Z)-2a\cov(\C{N}, F) \notag\\
&=-\EE(\C{N})^2 +a^2(\var(F)+\var(Z))+2a \EE(\C{N}) \EE(F)+\EE(\C{N}^2)-2a\EE(\C{N}F) \notag \\
&=\frac{32}{9}-\partial_{x x} \Psi(-1,0;1)+\partial_{x x} \Psi(0,1;
1)-\frac{8}{3}\{\partial_{x y} \Psi(-1,0;1)-\partial_{xy} \Psi(0,1;1)\} \notag \\&=
\frac{32}{9}+\frac{2}{3}\partial_x \Psi(0,0;1)=\frac{32}{9}-\frac{2}{3}\int_0^\infty e^{-t} (1+t)^{-1}\, dt.
\end{align}
For the third equality, we use the formula for the moment generating function of $(\C{N}, F)$, given in Theorem \ref{MGFProposition}, and \eqref{Derivative01}, \eqref{Derivative10}. 

\nid The fourth equality is true because by \eqref{abrec}, 
$$\Psi(x-1, y;1)-\Psi(x, y+1;1)=(x-y)\Psi(x, y; 1),$$
so that
\begin{align*}\partial_{x x} \Psi(-1,0;1)-\partial_{x x} \Psi(0,1;
1) &=
\partial_{x x}\{(x-y)\Psi(x, y;1)\}|_{x=y=0}=2\partial_x \Psi(0, 0;1),\\
\partial_{x y} \Psi(-1,0;1)-\partial_{xy} \Psi(0,1;1)&=
\partial_{x y}\{(x-y)\Psi(x, y;1)\}|_{x=y=0}\\&=-\partial_x \Psi(0, 0;1)+\partial_y \Psi(0, 0;1)=-\partial_x \Psi(0, 0;1).
\end{align*}
We used  \eqref{zero} in the last equality. The last equality in \eqref{VarComputation} follows from  \eqref{Derivative00}.

\section{Proof of Corrolary \ref{frequency}} \label{elementaryProof}

First we prove \eqref{meanDensity}.  We use the notation of Section \ref{ProofCLT}. For $n\ge1$,
$$N[0, S_n]=N[0, S_0)+1+\sum_{k=1}^n \C{N}_k.$$
Thus,
$$\lim_{n\to\infty} \frac{N[0, S_n]}{S_n}=\lim_{n\to\infty} \frac{N[0,
S_n]/n}{S_n/n}=\frac{\EE(\C{N})}{\EE(X_1)}=\frac{2}{3}\EE(\C{N}).$$
Since the process $(N[0, S_n])_{n\ge1}$ is increasing in $n$ and
$\lim_{n\to\infty} S_{n+1}/S_n=1$, with interpolation we get that
\begin{equation}\label{CorrolaryLimit}\lim_{n\to\infty} \frac{N[0, t]}{t}=\frac{2}{3}\EE(\C{N}).
\end{equation}
The proof of \eqref{meanDensity} is completed by noting that $\EE(\C{N})=2$ because of Theorem
\ref{MGFProposition}. However, since the proof of that theorem is quite
involved, we give in the following section an easy proof of $\EE(\C{N})=2$.

For $a>0$, the stationarity of $\xi$ and the ergodic theorem give that for $n\to\infty$,
\begin{equation}
\frac{N[0, na)}{n}=\frac{1}{n}\sum_{k=1}^n N[(k-1)a, ka)\to G
\end{equation}
a.s. and in $L^1$, where $G$ is a random variable.  
By \eqref{CorrolaryLimit}, $G=(4/3)a$, and since $\EE N[0, na)=n\EE N[0, a)$, the $L^1$ convergence  gives that $\EE N[0, a)=(4/3)a$. Now the stationarity of $\xi$ together with standard arguments show that $\EE N(A)=(3/4)\gl(A)$ for each Borel $A\subset\D{R}$.

\section{The expected value of $\C{N}$} \label{expComputation}
Although the expectation of $\C{N}$ was computed in Theorem
\ref{MGFProposition}, here we give an
alternative, elementary derivation based on a double counting argument. 

As noted in the proof of Proposition \ref{LaplTrans} (Subsection
\ref{LaplaceSubsection}), $\C{N}$ has the same law
as $\C{N}(\gep, 1)$ where $\gep$ is an excursion with law $n(\cdot\, |
\,
\ol{\gep}\ge 1)$. Expectation with respect to this law will be
denoted by $\EE_n( \cdot \, | \, \ol{\gep}\ge 1)$.

\begin{lemma}
 $\EE_n(\C{N}(\gep, 1) \, | \, \ol{\gep}\ge 1)=2$.
\end{lemma}
\begin{proof}
For the path $B$ with $B|(-\infty, 0)=+\infty$ and $B|[0,+\infty)$ a standard
Brownian motion, we define the process $x_B$ exactly as in the introduction. Now $x_B$ is an increasing function, it moves always forward to
deeper and deeper valleys of $B$. We will count in two ways the
number $\nu^+(x)$ of jumps of $x_B$ in $[1,x]$.

\smallskip

\textsc{First way}:

\smallskip

\nid Let $T_0=0, h_0:=1$, and define (see Figure \ref{fig2})
\begin{align*}\gs_1:=&\min\{s>0:B_s-\ul{B}_s=h_0\}, \\
R_1:=&-\ul{B}_{\gs_1}, \\
\tau_1:=&\min\{s:B_s=-R_1\},\\
h_1:=& \ol{B}_{\tau_1}-\ul{B}_{\tau_1},\\
T_1:=&T_0+\tau_1.
\end{align*}
\begin{figure}[htbp]
\begin{center}
\resizebox{10cm}{!} {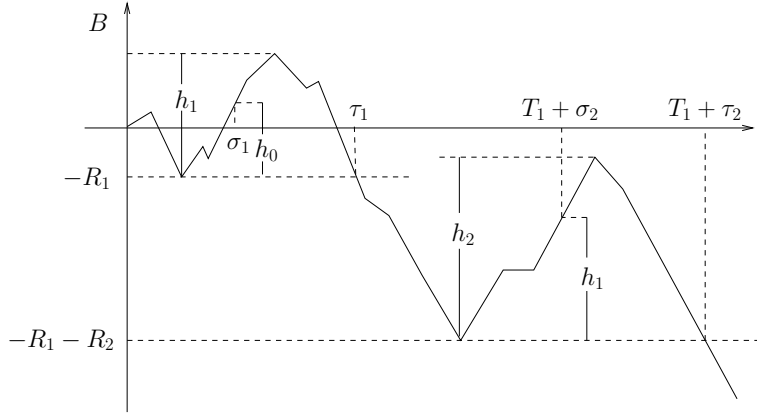} \caption{First decomposition of a one sided Brownian motion path.}
\label{fig4}
\end{center}
\end{figure}
We repeat the same procedure for the process
$(B_{s+T_1}-B_{T_1})_{s\ge0}$ with the roles of $h_0, T_0$ played
now by $h_1, T_1$. Thus we define $\gs_2, \tau_2, h_2, R_2, T_2$, and we
continue recursively.

Using the strong Markov property and an argument analogous to the one in the
proof of Lemma \ref{overshoots}, we see that the random variables
$w_n:=h_n/h_{n-1}, n\ge1$ are i.i.d. and each has density $x^{-2}
1_{x\ge1}$. In particular, $\log w_1$ has the exponential distribution with
mean 1. Note that $h_n=\prod_{i=1}^n w_i$, $\nu^+(h_n)=n$, so that
$$\frac{\nu^+(h_n)}{\log
h_n}=\frac{n}{\sum_{i=1}^{n-1}\log w_i}.$$
By the law of large numbers, this converges to
1 since $\EE(\log w_1)=1$. With interpolation we show
that 
\begin{equation}\label{limitNplus}\lim_{x\to+\infty}\frac{\nu^+(x)}{\log
x}=1.\end{equation}

\medskip

\textsc{Second way:} Now we split the path of $B$ using a
different strategy. By analogy with Section \ref{structuresection}, we define 
\begin{align*}
H_\ell^+:=& \inf \{s>0: B_s=\ell\}, \\
\Theta_\ell^+:=&-\min \{B_s: s\in[0, H_\ell^+]\}.
\end{align*}
\nid Again we can see that $\Theta^+$ is piecewise constant, left continuous, and the set of points where it jumps, call it $\C{L}^+$, has 0 as only accumulation
point. 

\nid Pick $\ell\in\C{L}^+$. At $\ell$, $\Theta^+$ jumps because at height
$\ell$, the first excursion
of $\ol{B}-B$ appeared which goes deeper
than $-\Theta_{\ell}^+$. Let (see Figure \ref{fig5})
$$h(\ell):=\ell+\Theta_\ell^+.$$
\begin{figure}[htbp]
\begin{center}
\resizebox{10cm}{!} {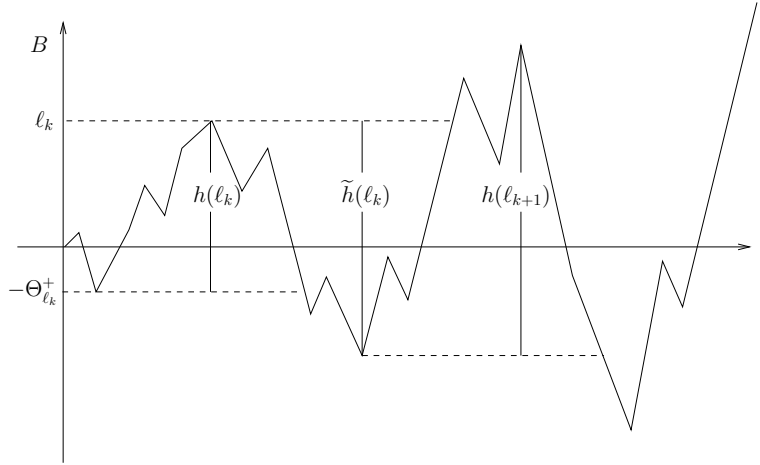} \caption{Second  decomposition of a one sided Brownian motion path.}
\label{fig5}
\end{center}
\end{figure}
The aforementioned excursion has law
$n(\,\cdot\, \, | \, \ol{\gep}\ge h)$. Call $\widetilde h(\ell)$ its
height.
Number the elements of the set $\C{L}^+\cap[1, \infty)$ in increasing
order as
$(\ell_n)_{n\ge1}$, and for each $n$, call $\gep_n$ the excursion that gives
rise to the jump at $\ell_n$. 
Then, with the same arguments as in Lemma \ref{overshoots}, we can
prove the
following.

\smallskip

\nid \textsc{Claim}: The random variables
$$\left\{\frac{\widetilde h(\ell_k)}{h(\ell_k)}, \frac{h(\ell_{k+1})}{\widetilde
h(\ell_k)}:k\ge1\right\}$$
are i.i.d., and each has density $x^{-2}1_{x\ge1}$.

\smallskip

Now note that
\begin{align*}\nu^+(h(\ell_n))&=\nu^+(h(\ell_1))-1+\C{N}(\gep_1,
h(\ell_1))+\C{N}(\gep_2, h(\ell_2))+\cdots+\C{N}(\gep_{n-1}, h(\ell_{n-1})),\\
h(\ell_n)&=h(\ell_1)\prod_{k=1}^{n-1} \frac{h(\ell_{k+1})}{\widetilde h(\ell_k)}
\frac{\widetilde h(\ell_k)}{h(\ell_k)}.
\end{align*}
The above claim gives that for each $k\ge1$, the random variables
$\log(h(\ell_{k+1})/\widetilde h(\ell_k),  \log(\widetilde h(\ell_k)/h(\ell_k))$
are exponential with mean 1, so that
\begin{equation}\label{NplusLimit}\lim_{n\to\infty} \frac{\nu^+(h(\ell_n)}{\log
h(\ell_n)}
=\frac{\EE_n(\C{N}(\gep, 1)\, | \, \ol{\gep}\ge
1)}{2}
\end{equation}
The result follows by comparing \eqref{limitNplus}, \eqref{NplusLimit}.
\end{proof}

\textbf{Acknowledgments}: I thank Balint Virag for useful discussions. 

\bibliographystyle{plainnat}

\begin{thebibliography}{}


	
\bibitem[Bertoin\hspace{-0.5ex}(1996)]{BE} Bertoin, J. \textit{L\'evy processes, volume
121 of Cambridge
Tracts in Mathematics.} (1996).


 \bibitem[Bovier\hspace{-0.5ex}(2006)]{BO} Bovier, A., \textit{ Metastability: a potential
theoretic approach}. In Proceedings of the ICM 2006, Madrid,  499-518, European
Mathematical Society. 

\bibitem[Cheliotis\hspace{-0.5ex}(2005)]{C} Cheliotis, D.,
   \textit{Difusion in random environment and the renewal theorem}. Ann. Probab.
   \textbf{33}, no. 5:1760-1781, (2005).

\bibitem[Cheliotis and Virag \hspace{-1ex} (2013)]{CV} Cheliotis, D. and Virag, B. \textit{Patterns in
Sinai's walk.} Ann.
Probab., \textbf{41} (3B), 1900-1937, (2013). 

\bibitem[Daley and Vere-Jones \hspace{-1ex}(2003)]{DVJ} Daley, D.J., Vere-Jones, D. \textit{An introduction
to the theory
of point processes. Vol. I: Elementary theory and methods.} Second Edition, Springer, (2003).


\bibitem[Le Dousal et al.\hspace{-0.5ex} (1999)]{DMF}
 Le Doussal, P.,  Monthus, C. and Fisher, D.
  \textit{Random walkers in
one-dimensional random environments: Exact renormalization group
analysis.} Physical Review E, \textbf{59} (5), 4795-4840, (1999).

\bibitem[Durrett\hspace{-0.5ex}(2010)]{DU} Durrett, R. \textit{Probability: Theory and
Examples.} Cambridge University Press, Fourth edition, (2010).


\bibitem[Lebedev\hspace{-0.5ex}(1972)]{LE} Lebedev, N. N. \textit{Special functions and their
applications.} Revised edition, translated from the Russian and edited by
Richard A. Silverman. Unabridged and corrected republication. (1972).

\bibitem[Polyanin and Zaitsev \hspace{-0.5ex}(2003)]{PZ} Polyanin, A. and Zaitsev, V. \textit{Handbook of
exact solutions
for ordinary differential equations}. Chapman \& Hall,
Second edn. (2003)


\bibitem[Neveu and Pitman \hspace{-0.5ex}(1989)]{NP} Neveu, J. and Pitman, J., \textit{Renewal property
of the
extrema and tree property of the excursion of a one-dimensional
{B}rownian motion} in S\'{e}minaire de Probabilit\'{e}s XXIII,
vol. 1372, Lecture Notes in Math., Springer, Berlin, 239-247,
(1989).

\bibitem[Revuz and Yor(1999)]{RY} Revuz, D. and Yor, M., \textit{Continuous
martingales and Brownian
motion},
 3rd edn, Springer, Berlin,(1999).

\bibitem[Seignourel\hspace{-0.5ex}(2000)]{SEI} Seignourel, P.,\textit{Discrete schemes for
processes in random
media.} Probab. Theory Relat. Fields, \textbf{118} (3), 293-322",
(2000).

\bibitem[Shi\hspace{-0.5ex}(2001)]{SHI} Shi, Z. \textit{Sinai's walk via stochastic
calculus}, in F. Comets and E. Pardoux, editors, ``Milieux
Al\'{e}atoires, Panoramas et Synth\`{e}ses '', vol. 12, ,
Soci\'{e}t\'{e} Math\'{e}matique de France, (2001).

\bibitem[Tanaka\hspace{-0.5ex}(1988)]{TA} Tanaka, H. \textit{Limit theorem for one-dimensional
diffusion process in Brownian environment}. In Stochastic Analysis. Springer
Berlin Heidelberg, 1988. 156-172.


\bibitem[Zeitouni\hspace{-0.5ex}(2004)]{ZE}  Zeitouni, O. \textit{Random walks in random
environment,
 Lectures on Probability Theory and Statistics Ecole d'Et\'{e} de
Probabilit\'{e}s de Saint-Flour XXXI-2001}, Lecture Notes in
Math., vol. 1837, Springer, (2004).

\end{thebibliography}

\end{document}